\documentclass[12pt, a4paper]{article}
\usepackage[english]{babel}
\usepackage{amssymb,amsmath,amsthm}
\usepackage{verbatim}
\usepackage[margin=2cm]{geometry}
\usepackage{stmaryrd}
\usepackage{authblk}
\theoremstyle{plain}

\newtheorem{theorem}{Theorem}
\newtheorem*{theorem*}{Theorem}
\newtheorem*{lemma*}{Lemma}
\newtheorem*{definition*}{Definition}
\newtheorem*{remark*}{Remark}

\newtheorem*{corollary*}{Corollary}

\newtheorem{proposition}[theorem]{Proposition}

\def\R{\mathbb R}
\def\Q{\mathbb Q}

\def\Z{\mathbb Z}
\def\F{\mathbb F}
\def\L{\mathcal{L}}
\def\ord{{\rm ord}}

\def\={\;=\;}
\def\+{\,+\,}
\def\-{\,-\,}
\def\lb{\llbracket}
\def\rb{\rrbracket}
\def\bal{\begin{aligned}}
\def\eal{\end{aligned}}
\def\be{\begin{equation}\label}
\def\ee{\end{equation}}

\AtEndDocument{\bigskip{\footnotesize%
  \textsc{Institute of Mathematics of the Polish Academy of Sciences} \par
  \'{S}niadeckich 8, 00-656, Warsaw  \par
  \texttt{m.vlasenko@impan.pl} \par
}}

\title{Formal groups and congruences}
\author{Masha Vlasenko\footnote{This work was supported by the National Science Centre of Poland, grant UMO-2016/21/B/ST1/03084.}}

\begin{document}
\date{}
\maketitle

We give a criterion of integrality of an one-dimensional formal group law in terms of congruences satisfied by the coefficients of the canonical invariant differential. For an integral formal group law a $p$-adic analytic formula for the local characteristic polynomial at $p$ is given. We demonstrate applications of our results to formal group laws attached to L-functions, Artin--Mazur formal groups of algebraic varieties and hypergeometric formal group laws. 

This paper was written with an intention to give an explicit and self-contained introduction to the arithmetic of formal group laws, which would be suitable for non-experts. By this reason we consider only one-dimensional laws, though a generalization of our approach to higher dimensions is clearly possible. The ideas of congruences and $p$-adic continuity in the context of formal groups were considered by many authors. We sketch the relation of our results to the existing literature in a separate paragraph at the end of the introductory section.     

\section{Introduction}\label{sec:intro}

Let $R$ be a commutative ring with the identity. A formal group law of dimension~1 over $R$ is a power series in two variables $F(x,y) \in R\llbracket x,y \rrbracket$ satisfying the conditions
\[\bal & F(x,0) \= F(0, x) \= x\,,\\
 & F(F(x,y),z) \= F(x,F(y,z))\,.
\eal\]
A formal group law is said to be commutative when $F(x,y)=F(y,x)$. For two formal group laws $F_1$ and $F_2$ over $R$ and a ring $R' \supseteq R$, a homomorphism $h \in {\rm Hom}_{R'}(F_1,F_2)$ is a power series $h \in x R'\lb x \rb$ such that $h(F_1(x,y))=F_2(h(x),h(y))$. An invertible homomorphism is called an isomorphism. An isomorphism $h$ is called strict if $h(x) \equiv x$ modulo degree $\ge 2$.    

From now on we assume that $R$ is a characteristic zero ring (i.e. $R \to R \otimes \Q$ is injective). In this case every formal group law $F \in R\lb x,y\rb$ is commutative and strictly isomorphic over $R \otimes \Q$ to the trivial formal group law ${\mathbb G}_a(x,y)=x+y$. The unique strict isomorphism $f \in {\rm Hom}_{R\otimes \Q}(F, {\mathbb G}_a)$ is called the logarithm of $F$. The logarithm satisfies 
\be{logF}
F(x,y)=f^{-1}(f(x)+f(y))
\ee
and can be written in the form 
\be{log_with_coefs}
f(x) = \sum_{n=1}^{\infty} \frac{b_{n-1}}n x^n
\ee
with $b_n \in R$, $b_0=1$. The shift in indices ($b_{n-1}$ instead of $b_n$) looks more natural in the view of the fact that $f'(x) dx = \sum_{n=0}^{\infty} b_n x^n \, dx$ is the canonical invariant differential on $F$ (\cite{Ho69},\cite[\S 5.8]{H}). We shall characterize those $R$-valued sequences $\{ b_n ; n \ge 0\}$ which arise as sequences of coefficients of canonical differentials on one-parameter formal group laws over $R$. This question is trivial when $R=R\otimes \Q$, in which case any sequence $\{ b_n ; n \ge 0\}$ yields a formal group law over $R$ by formulas~\eqref{logF} and~\eqref{log_with_coefs}. 

Fix a prime number $p$ and assume that $R$ is equipped with a $p$th power Frobenius endomorphism $\sigma: R \to R$ (i.e. $\sigma(r) \equiv r^p \mod pR$). We extend $\sigma$ to polynomials $R[x]$ and power series $R\lb x \rb$ by assigning $\sigma(x)=x^p$. Our first result (Theorem~\ref{cong_fg_thm} below) gives a necessary and sufficient condition for the sequence $\{ b_n ; n \ge 0\}$ under which $p$ doesn't show up in the denominators in~\eqref{logF}. In order to state this criterion we need to define a transformation of the sequence $\{ b_n ; n \ge 0\}$. For a non-negative integer $n$ we will denote by $\ell(n) = \min\{s \ge 1: n < p^s\}$ the length of the $p$-adic expansion of $n$. For two non-negative integers $n,m$ we denote by $n * m$ the integer whose $p$-adic expansion is the concatenation of the $p$-adic expansions of $n$ and $m$ respectively, that is $n * m = n + m \, p^{\ell(n)}$. Notice that $n * 0=n$ and $\ell(n * m)=\ell(n)+\ell(m)$ if and only if $m>0$.

\begin{definition*} Let $\{ b_n; n \ge 0\}$ be a sequence of elements of $R$. The \emph{$p$-sequence} associated to $\{ b_n; n \ge 0\}$ is the sequence of elements of $R$ given by
\[
c_n \= \sum_{\begin{tiny}\bal n&=n_1*\ldots *n_k \\ n_1 \ge & \, 0, n_2,\ldots,n_k > 0 \eal\end{tiny}}  (-1)^{k-1} \; b_{n_1} \cdot \sigma^{\ell(n_1)}(b_{n_2})\cdot \sigma^{\ell(n_1)+\ell(n_2)}(b_{n_3}) \cdot \ldots \cdot \sigma^{\ell(n_1)+\ldots+\ell(n_{k-1})}(b_{n_k}) 
\]
for $n \ge 0$, where the sum runs over all possible decompositions of the $p$-adic expansion of $n$ into a tuple of $p$-adic expansions of non-negative integers.
\end{definition*}

\noindent The original sequence $\{ b_n; n \ge 0\}$ can be reconstructed from its $p$-sequence, hence any $R$-valued sequence can occur as a $p$-sequence associated to an $R$-valued sequence (see Section~\ref{sec:integrality} for details). To author's knowledge, $p$-sequences were first introduced by Anton Mellit in \cite{MV} motivated by the result which we will reproduce later (Proposition~\ref{Dwork_cong_prop} in Section~\ref{sec:AM}) in a slightly more general form.  
  
\bigskip

\begin{theorem}\label{cong_fg_thm} Let $R$ be a characteristic zero ring (i.e. $R \to R \otimes \Q$ is injective) endowed with a $p$th power Frobenius endomorphism $\sigma: R \to R$ (i.e. $\sigma(r) \equiv r^p \mod pR$ for every $r \in R$). Let $\{b_n ; n \ge 0\}$ be a sequence of elements of $R$ with $b_0=1$. Put $f(x) \= \sum_{n=1}^{\infty} \frac{b_{n-1}}{n} \, x^n$. The formal group law $F(x,y)=f^{-1}(f(x)+f(y))$ has coefficients in $R \otimes \Z_{(p)}$ if and only if the $p$-sequence $\{ c_n ; n \ge 0\}$ associated to $\{ b_n ; n \ge 0\}$ satisfies
\be{c_cong_0}
c_{mp^k-1} \in p^{k} R \quad\text{ for all }\quad m > 1, k\ge 0\,.
\ee
\end{theorem}

\bigskip

The reader may notice that the statement of the theorem is trivial when $p$ is invertible in $R$, in which case $R \otimes \Z_{(p)} = R \otimes \Q$. Since $\cap_p (R \otimes \Z_{(p)}) = R$, the following global criterion follows immediately.

\begin{corollary*} Let $R$ be a characteristic zero ring endowed with a $p$th power Frobenius morphism for every rational prime $p$. Let $\{ b_n; n \ge 0\}$ be a sequence of elements of $R$ with $b_0=1$. In the notation of Theorem~\ref{cong_fg_thm}, we have $F(x,y) \in R\lb x,y \rb$ if and only if for every $p$ the respective $p$-sequence $\{ c_n ; n \ge 0\}$ (it is a different one for each $p$) associated to $\{ b_n; n \ge 0\}$ satisfies~\eqref{c_cong_0}.
\end{corollary*} 

We prove Theorem~\ref{cong_fg_thm} in Section~\ref{sec:integrality}. The proof is based on Hazewinkel's functional equation lemma (\cite[\S 2.2]{H}). By~\cite[Proposition 20.1.3]{H} every formal group law over a $\Z_{(p)}$-algebra is \emph{of functional equation type}. In our case it means that $F(x,y)$ has coefficients in $R \otimes \Z_{(p)}$ if and only if there exists a sequence of elements $v_1,v_2,\ldots \in R \otimes \Z_{(p)}$ such that the series
\be{func_eq_data}
g(x) \= f(x) \- \frac1{p} \sum_{s=1}^{\infty} v_s \cdot (\sigma^s f)(x) 
\ee
has coefficients in $R \otimes \Z_{(p)}$. Such a sequence $\{ v_s; s \ge 1 \}$ is non-unique. Our proof shows that one of possible choices is given by $v_s = \frac{c_{p^s-1}}{p^{s-1}} \in R$.   

\bigskip

From now on let us assume that $R$ is the ring of integers of a complete absolutely unramified discrete valuation field of characteristic zero and residue characteristic $p>0$, equipped with a lift $\sigma: R \to R$ of the $p$th power Frobenius on the residue field $R/pR$. In this case one can construct the \emph{shortest possible} functional equation~\eqref{func_eq_data} for the logarithm $f(x)$ of a formal group law $F(x,y) \in R\lb x,y \rb$. Namely, suppose we have any functional equation~\eqref{func_eq_data} and let 
\be{height_def}
h \= \inf\{ s \ge 1 \;:\; v_s \in R^{\times} \} \,.
\ee  
If $h=\infty$ then $f(x) \in R\lb x \rb$ and hence $F$ is strictly isomorphic to $\mathbb{G}_a$ over $R$. If $h<\infty$ then there exist unique $\alpha_1,\ldots,\alpha_{h-1} \in pR$ and $\alpha_h \in R^{\times}$ such that
\[
f(x) \- \frac1{p} \sum_{s=1}^{h} \alpha_s \cdot (\sigma^s f)(x) \in R\lb x \rb\,. 
\]
The Eisenstein polynomial
\[
\Psi_F(T) \= p \- \sum_{i=1}^h \alpha_i T^i  \;\in\; R[T] 
\]
is called the \emph{characteristic polynomial} of $F(x,y)$ and $h$ is called the \emph{height}. Two formal group laws over $R$ are strictly isomorphic if and only if their characteristic polynomials are equal (\cite[Proposition 3.5]{Ho69}, \cite[Theorem 20.3.12]{H}). By convention, we put $\Psi_F(T)=p$ if $h = \infty$.

In the case when the residue field is finite ($\# R / pR = q = p^f$) the above defined height coincides with the height of the multiplication by $p$ endomorphism $[p]_{\overline{F}}$ of the reduction $\overline F(x,y) \in \F_q \lb x,y \rb$. However $\Psi_F(T)$ shall not be confused with the characteristic polynomial of the Frobenius endomorphism $\xi_{\overline F}(x)=x^q$. When $R=\Z_p$ then $\alpha_h^{-1} \Psi_F(T)$  coincides with the characteristic polynomial of $\xi_{\overline F}$, but in general the relation between these two invariants is more subtle (see \cite[\S 30.4, Remark 18.5.13]{H}).

The following theorem describes congruences satisfied by the coefficients of the logarithm of an integral formal group law and provides a $p$-adic analytic formula for the characteristic polynomial. 

\bigskip

\begin{theorem}\label{P_formulas_thm} Let $R$ be the ring of integers of a complete absolutely unramified discrete valuation field of characteristic zero and residue characteristic $p>0$, equipped with a lift $\sigma: R \to R$ of the $p$th power Frobenius on the residue field $R/pR$. Let $F(x,y) \in R\lb x,y \rb$ be a formal group law of height $h$. When $h < \infty$ we denote by $\Psi_F(T) = p - \sum_{i=1}^h \alpha_i T^i \in R[T]$  the characteristic polynomial of $F$. We write the logarithm of $F$ in the form
\[
\log_F(x) \= \sum_{n=1}^{\infty} \frac{b_{n-1}}{n} x^n\,,
\] 
where $\{ b_n; n \ge 0\}$ is a sequence of elements of $R$ with $b_0=1$. 
\begin{itemize}
\item[(i)] If $h \= \infty$ then $\ord_p(b_{p^n-1}) \ge n$ for all $n$.

 If $h < \infty$ then $\ord_p(b_{p^n-1}) \ge n - \lfloor\dfrac n h \rfloor$ with equality 
when $h | n$.

\item[(ii)] When $h<\infty$, consider elements $\beta_n = b_{p^n-1} / p^{n - \lfloor\frac n h \rfloor}\in R$. We have 
\[
\beta_{kh} \equiv \prod_{i=0}^{k-1} \sigma^{ih}(\beta_h) \mod p\,.
\]

\item[(iii)] When $h<\infty$, consider  for every $k \ge 1$  an $h \times h$ matrix with coefficients in $R$ given by 
\[
D_k \= \Bigl(p^{\varepsilon_{ij}} \sigma^{j+1}(\beta_{kh-1+i-j}) \Bigr)_{0 \le i,j \le h-1} \quad\text{ with }\quad \varepsilon_{ij} \= \begin{cases} 0\,, &j<i \text{ or } j=h-1\,,\\ 1\,, &i \le j < h-1\,.\end{cases} 
\]
We have that $\det D_k \equiv (-1)^{h-1} \, \prod_{s=1}^{kh-1}  \sigma^s(\beta_h)\ne 0 \mod p$ and
\be{padic_formula}
 \begin{pmatrix} \alpha_1/p \\ \alpha_2/p \\ \vdots \\ \alpha_{h-1}/p \\ \alpha_h \end{pmatrix} \; \equiv \;  D_k^{-1} \begin{pmatrix} \beta_{kh}\quad\;\; \\ \beta_{kh+1}\quad \\ \vdots \\ \beta_{kh+h-2} \\ \beta_{kh+h-1} \end{pmatrix}  \mod p^k\,.
\ee
\end{itemize}
\end{theorem}
 
\bigskip  
  
For example, it follows from (i) that the height is equal to $1$ if and only if $p \not | b_{p-1}$. In this case we have that $p \not | b_{p^k-1}$ for all $k \ge 1$ and there exists a unique unit $\alpha_1 \in R^\times$ such that for every $k \ge 1$
\be{unit_root}
b_{p^k-1}/\sigma(b_{p^{k-1}-1}) \;\equiv\; \alpha_1 \; \mod p^k \,.
\ee
The characteristic polynomial is given by $\Psi_F(T) \= p \- \alpha_1 \, T$. 

In the case of height~$2$, we have $\ord_p(b_{p^{2k}-1})=k$ and $\ord_p(b_{p^{2k-1}-1})\ge k$ for all $k$ by~(i), and~(iii) gives us the following formulas for the coefficients of the characteristic polynomial $\Psi_F(T) \= p - \alpha_1 T - \alpha_2 T^2$:
\[
\begin{pmatrix} \frac{\alpha_1}p \\ \alpha_2 \end{pmatrix} \; \equiv \;  \begin{pmatrix} p \, \frac{\sigma(b_{p^{2k-1}-1})}{p^k} &  \frac{\sigma^2(b_{p^{2k-2}-1})}{p^{k-1}} \\  \frac{\sigma(b_{p^{2k}-1})}{p^k} &  \frac{\sigma^2(b_{p^{2k-1}-1})}{p^k}  \end{pmatrix}^{-1} \begin{pmatrix}  \frac{b_{p^{2k}-1}}{p^k} \\  \frac{b_{p^{2k+1}-1}}{p^{k+1}}\end{pmatrix} \mod p^k
\]
or, equivalently, 
\be{ht_2}\bal
&\alpha_1 \;\equiv \; \quad \frac{\sigma^2(b_{p^{2k-1}-1})\, b_{p^{2k}-1}-\sigma^2(b_{p^{2k-2}-1})\, b_{p^{2k+1}-1}}{\sigma(b_{p^{2k-1}-1})\sigma^2(b_{p^{2k-1}-1}) - \sigma(b_{p^{2k}-1})\sigma^2(b_{p^{2k-2}-1})} \mod p^k \,,\\
&\alpha_2 \;\equiv\; \frac1p \; \frac{\sigma(b_{p^{2k-1}-1})\,b_{p^{2k+1}-1}-\sigma(b_{p^{2k}-1})\,b_{p^{2k}-1}}{\sigma(b_{p^{2k-1}-1})\sigma^2(b_{p^{2k-1}-1}) - \sigma(b_{p^{2k}-1})\sigma^2(b_{p^{2k-2}-1})} \mod p^k \,.
\eal\ee  

\bigskip

We prove Theorem~\ref{P_formulas_thm} in Section~\ref{sec:local_invariants}. Section~\ref{sec:applications} is devoted to applications of our theorems. We start with formal group laws attached to L-functions. In this case the $p$-sequence recovers coefficients of the respective local $L$-factor at $p$ and there are finitely many non-zero congruences~\eqref{c_cong_0} to be checked. In Section~\ref{sec:AM} we show that Theorem~\ref{cong_fg_thm} implies integrality of certain Artin--Mazur formal groups arising from cohomology of algebraic varieties. In this situation Theorem~\ref{P_formulas_thm} is useful for computation of eigenvalues of the Frobenius operator on the respective crystalline (or $\ell$-adic) cohomology group. In Section~\ref{sec:HG} we give a criterion of integrality and compute local characteristic polynomials of hypergeometric formal group laws generalizing the results in~\cite{Ho72_HG}. Here again the criterion in Theorem~\ref{cong_fg_thm} can be reduced to finitely many congruences.

\bigskip

{\bf Acknowledgement.} The criterion of integrality (Theorem~\ref{cong_fg_thm}) was initially conjectured by Eric Delaygue during our correspondence. I would like to thank Eric for numerous fruitful discussions of the subject. I am also grateful to Piotr Achinger and Susanne M\"{u}ller, whose remarks helped to improve the exposition.

\bigskip

{\bf Relation to the existing literature. } The approach to integrality of formal group laws via functional equations is classical. Its grounds were laid by Taira Honda in~\cite{Ho69}, and it is expressed in the most general form by the Hazewinkel's functional equation lemma and criterion of $p$-integrality (\cite[\S 2, 20]{H}). Our Theorem~1 can be viewed as an algorithmic version of this approach: we do a transform of the sequence of coefficients of the invariant differential (we call this transformed sequence the $p$-sequence), check the respective congruences and recover one of the functional equations whenever they exist. Though the $p$-sequence is a new combinatorial object, it arises naturally in the above context. A similar idea was used by Bert Ditters for the classification of commutative formal group laws over $p$-Hilbert domains: compare~\cite[(1.2) and (1.9)]{Dit89} to an equivalent definition of the $p$-sequence by formula~\eqref{p_seq} below. 

Our Theorem~2 is an explicit and general statement of the Atkin and Swinnerton-Dyer congruences for formal group laws. Its proof is rather straightforward and based on Honda's methods, as we explain in Section~\ref{sec:local_invariants}. The characteristic polynomial of a $p$-integral formal group law arises when one looks for ``the shortest possible functional equation''. Since a functional equation can be translated to the relation between the Frobenius and Verschiebung in the respective Cartier--Dieudonn\'{e} module (see e.g.~\cite{Dit89} and~\cite[Chapter V]{H}), this way one also recovers the minimal polynomial satisfied by the Frobenius endomorphism of the reduction of the formal group law modulo $p$.

\section{A criterion of integrality of a formal group law}\label{sec:integrality}

This section is devoted to the proof of Theorem~\ref{cong_fg_thm}. We recall that $R$ is a ring of characteristic zero endowed with a morphism $\sigma \in {\rm End}(R)$ satisfying $\sigma(r) \equiv r^p \mod pR$ for any $r \in R$. We extend $\sigma$ to a $p$th power Frobenius morphism of $R \otimes \Q$ by $\Q$-linearity and further to the ring of power series in $x$ with coefficients in $R \otimes \Q$ by assigning $\sigma(x)=x^p$. For a sequence $\{ b_n ; n \ge 0\}$ of elements of $R$ the respective $p$-\emph{sequence} was defined as
\be{p_seq_def}
c_n \= \sum_{\begin{tiny}\bal n&\=n_1*\ldots *n_k \\ n_1 \ge & \, 0, n_2,\ldots,n_k > 0 \eal\end{tiny}}  (-1)^{k-1} \; b_{n_1} \cdot \sigma^{\ell(n_1)}(b_{n_2})\cdot \sigma^{\ell(n_1)+\ell(n_2)}(b_{n_3}) \cdot \ldots \cdot \sigma^{\ell(n_1)+\ldots+\ell(n_{k-1})}(b_{n_k})\,.
\ee
When $k=1$ we have $n=n_1 \ge 0$ and the condition $n_2,\ldots,n_k>0$ is empty. This is a finite sum because $k \le \ell(n)$. We will start with discussing some properties of $p$-sequences. For small indices we have 
\[\bal
&c_0\=b_0\,,\quad c_1\=b_1\,,\quad \ldots\,, \quad c_{p-1}\=b_{p-1}\,,\quad \\
&c_p\= b_p \- b_0 \, \sigma(b_1)\,,\quad c_{1+p}\= b_{1+p} \- b_1 \, \sigma(b_1)\,, \quad \ldots \\
&c_{p^2}\= b_{p^2} \- b_0 \, \sigma(b_p)\,,\quad c_{1+p^2}\= b_{1+p^2} \- b_1 \, \sigma(b_p)\,, \quad \ldots \quad \\
& c_{p+p^2} \= b_{p+p^2} \- b_0 \, \sigma(b_{1+p}) \- b_p \, \sigma^2(b_1)\+ b_0 \, \sigma(b_1) \sigma^2(b_1)\,,\quad \ldots \\
\eal\]
One can easily notice that the original sequence $\{ b_n; n\ge0\}$ can be reconstructed from its $p$-sequence $\{ c_n; n\ge0\}$ since the right-hand side of~\eqref{p_seq_def} is the sum of $b_n$ and an expression containing only $b_k$ with $k<n$. 

\begin{lemma*} We have $b_0=c_0$ and 
\[
b_n \= \sum_{\begin{tiny}\bal n \= &n_1*(  \ldots *(n_{k-1}*n_k)\ldots) \\ & n_k > 0 \eal\end{tiny}} \; c_{n_1} \cdot \sigma^{\ell(n_1)}(c_{n_2})\cdot \sigma^{\ell(n_1)+\ell(n_2)}(c_{n_3}) \cdot \ldots \cdot \sigma^{\ell(n_1)+\ldots+\ell(n_{k-1})}(c_{n_k})
\]
for every $n > 0$.
\end{lemma*}
\begin{proof}Observe that
 \[\bal
c_n &\= \sum_{\begin{tiny}\bal n&\=n_1*\ldots *n_k \\ n_1 \ge & \, 0, n_2,\ldots,n_k > 0 \eal\end{tiny}}  (-1)^{k-1} \; b_{n_1} \cdot \sigma^{\ell(n_1)}(b_{n_2})\cdot \sigma^{\ell(n_1)+\ell(n_2)}(b_{n_3}) \cdot \ldots \cdot \sigma^{\ell(n_1)+\ldots+\ell(n_{k-1})}(b_{n_k}) \\
& \= \quad b_n \- \sum_{\begin{tiny}\bal n\=n_1 & *n_2 \\ n_2 > 0 \eal\end{tiny}} c_{n_1} \cdot \sigma^{\ell(n_1)}(b_{n_2})\,,
\eal\]
and therefore 
\[\bal
b_n &\= c_n \+ \sum_{\begin{tiny}\bal n &=n_1 *n_2 \\ &n_2 > 0 \eal\end{tiny}} c_{n_1} \cdot \sigma^{\ell(n_1)}(b_{n_2}) \\
&\= c_n \+ \sum_{\begin{tiny}\bal n &=n_1 *n_2 \\ & n_2 > 0 \eal\end{tiny}} c_{n_1} \cdot \sigma^{\ell(n_1)}(c_{n_2}) \+ \sum_{\begin{tiny}\bal n=n_1 & * (n_2*n_3) \\ & n_3 > 0 \eal\end{tiny}} c_{n_1} \cdot \sigma^{\ell(n_1)}(c_{n_2}) \cdot \sigma^{\ell(n_1)+\ell(n_2)}(b_{n_3})  \\
& \= \ldots \,,\\
\eal\]
which yields the statement of the lemma by iteration. 
\end{proof}

In what follows we will often use the following relation between the $p$-sequence and the original sequence, which we extract from the proof of the above lemma:
\be{p_seq}
b_n \= c_n \+ \sum_{\begin{tiny}\bal n &=n_1 *n_2 \\ &n_2 > 0 \eal\end{tiny}} c_{n_1} \cdot \sigma^{\ell(n_1)}(b_{n_2}) \,.
\ee  

\begin{proof}[Proof of Theorem~\ref{cong_fg_thm}.] $\Leftarrow$ Assume that~\eqref{c_cong_0} holds, so we have
\be{c_cong}
c_{m p^k-1} \in p^k R \;\text{  for all }\; m>1, k \ge 0\,.
\ee
For $s\ge 1$ consider elements $v_s = \frac1{p^{s-1}} c_{p^s-1}$ which lie in $R$ due to~\eqref{c_cong}. Consider the power series
\be{func_eq}
\sum_{n=1}^{\infty} d_n x^n \= f(x) \- \frac1{p} \sum_{s=1}^{\infty} v_s \cdot (\sigma^s f)(x)\,.
\ee
For each $n$, write $n=mp^k$ with $k \ge 0$ and $(m,p)=1$. If $k=0$ then $d_n=\frac1{n} b_{n-1} \in R \otimes \Z_{(p)}$. If $k>0$ we have 
\[\bal
d_n &\= \frac1{m p^k} b_{mp^k-1}  \- \frac1p \sum_{i=1}^k v_i \cdot \frac1{m p^{k-i}} \sigma^i (b_{mp^{k-i}-1})  \\
&\= \frac1{m p^k} \Bigl( b_{mp^k-1} \- \sum_{i=1}^k c_{p^i-1} \cdot \sigma^i(b_{m p^{k-i}-1})   \Bigr)\\
&\= \frac1{m p^k} \underset{\begin{tiny}\bal m  &=  m' * m'' \\ &m'>1\eal\end{tiny}}{\sum} c_{m' p^k - 1} \cdot \sigma^{k+\ell(m')}  (b_{m''})\,,
\eal\]
where the sum is over all possible decompositions $m = m'+m''p^{\ell(m')}$ with $m'>1$. If $m=1$ the sum is simply $0$. If $m>1$ then we have $c_{m'p^k-1} \in p^k R$ for every term in the sum due to~\eqref{c_cong}, and therefore $d_n \in R \otimes \Z_{(p)}$. We proved that~\eqref{func_eq} is a series  with coefficients in $R \otimes \Z_{(p)}$, and therefore $F(x,y)$ has coefficients in the same ring by Hazewinkel's functional equation lemma (\cite[\S 2.2]{H}).

\bigskip

$\Rightarrow$ Suppose $F(X,Y)$ has coefficients in $R \otimes \Z_{(p)}$. By~\cite[Proposition 20.1.3]{H} every formal group law over a $\Z_{(p)}$-algebra is of functional equation type. This means that there exist elements $v_1,v_2,\ldots \in R \otimes \Z_{(p)}$ such that the series
\[
\sum_{n=1}^{\infty} d_n x^n \= f(x) \- \frac1{p} \sum_{s=1}^{\infty} v_s \cdot (\sigma^s f)(x) 
\]
has coefficients in $R \otimes \Z_{(p)}$. Formulated in terms of the initial $R$-valued sequence $\{ b_n \}$ this functional equation means that for every $k \ge 0$ and $m \ge 1$, $m$ not divisible by $p$, we have
\be{v_seq}
b_{m p^k - 1} \- \sum_{s=1}^k p^{s-1} \, v_s \,  \sigma^s(b_{m p^{k-s}-1}) \= m p^k d_{m p^k} \in p^k \, R \otimes \Z_{(p)}\,. 
\ee

Let's first prove that $c_{p^k-1} \in p^{k-1} R$ for all $k \ge 1$. We shall first prove that
\be{c_via_v_and_d}
c_{p^k-1} \= p^{k-1} \, v_k \+ p^k d_{p^k} \- \sum_{i=1}^{k-1} \sigma^{i}(c_{p^{k-i}-1})\, p^{i} \, d_{p^{i}} \,.
\ee
For $k=1$ this formula follows from $b_{p-1} \= c_{p-1} \= v_1 \+ p \, d_p$. We will now do induction in~$k$. Subtracting the two expressions
\[
b_{p^k-1} \= c_{p-1} \, \sigma(b_{p^{k-1}-1}) \+ c_{p^2-1} \, \sigma^2(b_{p^{k-2}-1}) \+ \ldots \+ c_{p^k-1} 
\]
and
\[
b_{p^k-1} \= v_1 \, \sigma(b_{p^{k-1}-1}) \+ p \, v_2 \, \sigma^2(b_{p^{k-2}-1}) \+ \ldots \+ p^{k-1} \, v_k \+ p^{k} \, d_k 
\]
(coming from~\eqref{p_seq} and~\eqref{v_seq} respectively)  and using~\eqref{c_via_v_and_d} for all indices smaller than $k$  we obtain that
\[\bal
c_{p^k-1} \- p^{k-1}& \, v_k \- p^k d_{p^k} \= \sum_{i=1}^{k-1} (p^{i-1} v_i - c_{p^i-1})\, \sigma^i(b_{p^{k-i}-1}) \\
&\= \sum_{i=1}^{k-1} \Bigl(- p^i d_{p^i} \+ \sum_{j=1}^{i-1} \sigma^{j}(c_{p^{i-j}-1})p^{j} d_{p^{j}}\Bigr)\, \sigma^i(b_{p^{k-i}-1})\\
&\= - \sum_{i=1}^{k-1} p^i d_{p^i} \Bigl( \sigma^i(b_{p^{k-i}-1}) \- \sum_{j=1}^{k-i} \sigma^{i}(c_{p^j-1}) \sigma^{i+j}( b_{p^{k-i-j}-1})\Bigr)\\
&\= - \sum_{i=1}^{k-1} p^i d_{p^i} \, \sigma^i(c_{p^{k-i}-1})\,,
\eal\]
which proves~\eqref{c_via_v_and_d} for this~$k$. The fact that $c_{p^k-1} \in p^{k-1} R \otimes \Z_{(p)}$ follows from~\eqref{c_via_v_and_d} by induction on $k$. It remains to observe that $R \cap (p^{k} R \otimes \Z_{(p)}) = p^k R$ for every $k$.

\bigskip

So far we proved~\eqref{c_cong} with $m=p$ for all $k$. Now consider the case of $m>1$ not divisible by $p$. Suppose that for some $k>0$ and for all such $m$ we knew that 
\be{diff}\bal
 b_{m p^k - 1} \- c_{p-1} \cdot \sigma(b_{m p^{k-1}-1}) \- c_{p^2-1} \cdot & \sigma^2(b_{m p^{k-2}-1}) \-  \ldots \- c_{p^k-1} \cdot \sigma^k(b_{m-1}) \\
& \= \underset{\begin{tiny}\bal m = m' * m''\\m'>1, m'' \ge 0\eal\end{tiny}}\sum c_{m' p^k - 1} \cdot \sigma^{k+\ell(m')}(b_{m''})
\eal\ee
belongs to $p^{k} R$. Then~\eqref{c_cong} for this $k$ would follow by induction on the length $\ell(m)$.
 
\bigskip

It remains to prove that the left-hand side in~\eqref{diff} belongs to $p^k R$. Since it obviously belongs to $R$ and $R \cap (p^{k} R \otimes \Z_{(p)}) = p^k R$, it is enough to show that the left-hand side in~\eqref{diff} belongs to $p^k R  \otimes \Z_{(p)}$.  Consider $\widetilde f(x) = \sum_{k=0}^{\infty}\frac{b_{p^k-1}}{p^k}x^{p^k}$ and the $p$-typical formal group law $\widetilde{F}(x,y)= {\widetilde f}^{-1}(\widetilde f(x)+\widetilde f(y))$. Observe that $\widetilde F$ has coefficients in $R \otimes \Z_{(p)}$ because we have the functional equation
\[
\widetilde f(x)\- \frac1p\sum_{s=1}^{\infty} v_s \cdot (\sigma^s \, \widetilde  f)(x) \= \sum_{k=0}^{\infty} d_{p^k} \, x^{p^k}\,.
\]
By~\cite[Theorem 16.4.14 and Remark 16.4.15]{H} $\widetilde F$ is isomorphic to $F$ over $R \otimes \Z_{(p)}$. Observe that we have another functional equation for $\widetilde f$: with $v_s'=\frac{c_{p^s-1}}{p^{s-1}} \in R$ one has
\[
\widetilde f(x)\- \frac1p\sum_{s=1}^{\infty} v_s' \cdot (\sigma^s \, \widetilde  f)(x) \= x\,. 
\]
Since $F$ and $\widetilde F$ are isomorphic, by part (iii) of Hazewinkel's functional equation lemma we must have that the coefficients of the power series $f(x)\- \frac1p\sum_{s=1}^{\infty} v_s' \cdot (\sigma^s f)(x)$ also lie in $R \otimes \Z_{(p)}$, which is precisely what we wanted to prove.
\end{proof}

\section{Computing local invariants}\label{sec:local_invariants}

In this section $R$ is the ring of integers of a complete absolutely unramified discrete valuation field $K$ of characteristic zero and residue characteristic $p>0$, equipped with a Frobenius endomorphism $\sigma: K \to K$ which is the lift of the $p$th power endomorphism on the residue field $R/pR$. The valuation is denoted by $\ord_p : K \to \Z \cup {+\infty}$.

Let $\{ v_s; s \ge 1\}$ be a sequence of elements of $R$ and $g \in x + x^2R\lb x \rb$ be a power series. The functional equation
\[
g(x) \= f(x)\- \frac1p \, \sum_{s=1}^{\infty} v_s \, (\sigma^s f)(x) 
\]
allows one to recover $f(x) \in x + x^2 K\lb x \rb$. Hazewinkel's functional equation lemma tells us that the formal group law $F(x,y)=f^{-1}(f(x)+f(y))$ has coefficients in $R$ (\cite[\S 2.2 (i)]{H}). Moreover, since $R$ is a $\Z_{(p)}$-algebra, every formal group law over $R$ can be constructed this way (\cite[Proposition 20.1.3]{H}). Two formal group laws with the same $\{ v_s; s \ge 1\}$ and different $g(x)$ are strictly isomorphic over $R$ (\cite[\S 2.2 (ii)]{H}). Conversely, if a formal group law over $R$ satisfies a functional equation with some $\{ v_s; s \ge 1\}$ then every formal group law strictly isomorphic to it over $R$ satisfies a functional equation with the same $\{ v_s; s \ge 1\}$ but different $g(x)$ (\cite[\S 2.2 (ii) and (iii)]{H}).  

Honda (\cite{Ho69},\cite[\S20.3]{H}) gave the following method to describe all sequences $\{ v_s; s \ge 1\}$ with which a given formal group law $F(x,y)=f^{-1}(f(x)+f(y))\in R\lb x,y\rb$ satisfies 
\be{func_eq_log}
f(x)\- \frac1p \, \sum_{s=1}^{\infty} v_s \, (\sigma^s f)(x) \;\in\; R\lb x \rb\,.
\ee 
(As we just explained, the set of such sequences corresponds to the strict isomorphism class of $F$.) Denote by $K_{\sigma}\lb T \rb$ the ring of 'non-commutative' power series with coefficients in $K$, where the multiplication satisfies $T a = \sigma(a) T$ for $a \in K$. We put the word \emph{non-commutative} in commas because this ring is commutative whenever $\sigma$ is the identity endomorphism, e.g. when $K=\Q_p$. We denote by $R_{\sigma}\lb T \rb \subset K_{\sigma}\lb T \rb$ the subring of power series with coefficients in $R$. With~\eqref{func_eq_log} one associates an element $\eta_v = p \- \sum_{s=1}^{\infty} v_s T^s \in R_{\sigma}\lb T \rb$. Then every other sequence $v' = \{ v'_s; s \ge 1\}$ with which~\eqref{func_eq_log} also holds comes as $\eta_{v'} = \theta \eta_v$ where $\theta \in 1 \+ R_{\sigma}\lb T \rb T$.

The height of $F(x,y)$ can be computed as
\be{h_via_v}
h = \inf\{ s \ge 1 \;:\; v_s \in R^{\times}\}
\ee
and is invariant under strict isomorphisms. If $h < \infty$, then by the formal Weierstrass preparation lemma (\cite[Lemma 20.3.13]{H}) there exist a unique power series $\theta \in 1 \+ R_{\sigma}\lb T \rb T$ and elements $\alpha_1, \ldots, \alpha_{h-1} \in p R$, $\alpha_h \in R^{\times}$ such that $\theta \eta_v = p \- \sum_{i=1}^h \alpha_i T^i$. The Eisenstein polynomial in the right-hand side of this expression is the characteristic polynomial of $F(x,y)$. It was denoted by $\Psi_F(T)$ in Section~\ref{sec:intro}. The strict isomorphism classes of formal group laws of finite height over $R$ correspond bijectively to such polynomials (\cite[Theorem 20.3.12]{H}).   

\bigskip

\begin{proof}[Proof of Theorem~\ref{P_formulas_thm}.] Let $\{ c_n; n \ge 0\}$ be the $p$-sequence associated with $\{ b_n ; n \ge 0\}$. By Theorem~\ref{cong_fg_thm} we have $c_{p^s-1} \in p^{s-1} R$ for all $s \ge 1$ and~\eqref{func_eq_log} is satisfied with the sequence $v_s \= \frac{c_{p^s-1}}{p^{s-1}} \in R$. We can now rewrite~\eqref{h_via_v} as
\be{h_via_c}
h = \inf\{ s \ge 1 \;:\; \ord_p(c_{p^s-1}) = s-1 \}\,.
\ee

 (i) Let's denote $\mu_n = b_{p^n-1}/p^n$, $\kappa_n = c_{p^n-1}/p^n$. Dividing the identity  
\be{bc_id}
b_{p^n-1} \= c_{p-1} \, \sigma(b_{p^n-1}) \+ c_{p^2-1} \, \sigma^2(b_{p^{n-1}-1}) \+ \ldots \+ c_{p^n-1} 
\ee
by $p^n$ we obtain
\be{mu_kappa_id}
\mu_n \= \kappa_1 \, \sigma(\mu_{n-1}) \+ \kappa_2 \, \sigma(\mu_{n-2}) \+ \ldots \+ \kappa_{n-1} \, \sigma^{n-1}(\mu_1) \+ \kappa_n\,. 
\ee
When $h = \infty$ then $\kappa_n \in R$ for all $n$ due to~\eqref{h_via_c}, and therefore $\mu_n \in R$ for all $n$ by induction. From now on we assume that $h<\infty$. We then have  $\ord_p(\kappa_i)\ge 0$ for $i<h$, $\ord_p(\kappa_h)=-1$ and $\ord_p(\kappa_i)\ge -1$ for $i>h$.
It follows that $\ord_p(\mu_i)\ge 0$ when $i<h$ and $\ord_p(\mu_h)=-1$. We will prove by induction that $\ord_p(\mu_n) \ge -\lfloor \frac nh \rfloor$ with equality when $h|n$. Suppose this inequality holds for all indices less than $n$. Then for $1 \le i <h$ we have $\ord_p(\kappa_i \, \sigma^i(\mu_{n-i}))=\ord_p(\kappa_i)+\ord_p(\mu_{n-i})\ge0-\lfloor \frac{n-i}h\rfloor \ge -\lfloor \frac nh \rfloor$ with strict inequality when $h | n$. For $i>h$ we have $\ord_p(\kappa_i\, \sigma^i(\mu_{n-i}))=\ord_p(\kappa_i)+\ord_p(\mu_{n-i})\ge -1-\lfloor \frac{n-i}h\rfloor = -\lfloor \frac{n-(i-h)}h\rfloor  \ge -\lfloor \frac nh \rfloor$, where the last inequality is again strict when $h|n$. When $i=h$ we have
\[
\ord_p(\kappa_h\, \sigma^h(\mu_{n-h}))=\ord_p(\kappa_h)+\ord_p(\mu_{n-h}) \begin{cases} \= -1 -\lfloor \frac {n-h}h \rfloor \= -\lfloor \frac nh \rfloor\, & \text{ when } h|n \,, \\
\;\ge\; -1 -\lfloor \frac {n-h}h \rfloor \= -\lfloor \frac nh \rfloor\, & \text{ when } h \not |n \,.\end{cases}
\]
Summarizing the above we get that $\ord_p(\mu_n) \ge -\lfloor \frac nh \rfloor$ with equality when $h|n$, which proves the first part of the theorem. 

(ii) We have $\beta_{kh}=p^k \mu_{kh}$. From~\eqref{mu_kappa_id} with $n=h$ we find that $p \mu_h \equiv p \kappa_h \mod p$.
Computation in part~(i) shows that modulo $p$ 
\[
p^{k} \mu_{kh} \equiv p^k \kappa_h \,\sigma^h( \mu_{(k-1)h}) \= (p \kappa_h) (p^{k-1}\, \sigma^h(\mu_{(k-1)h})) \equiv \ldots \equiv \prod_{i=0}^{k-1}(p \, \sigma^{ih}(\kappa_h)) \equiv \prod_{i=0}^{k-1}\sigma^{ih}(\beta_h) \,.
\]

(iii) The characteristic polynomial provides us with the shortest possible functional equation~\eqref{func_eq_log} where $v_s = \alpha_s$ for $s\le h$ and $v_s=0$ for $s > h$. Therefore for every $n \ge 0$
\[
p^n \quad | \quad b_{p^n-1} \- \sum_{s=1}^h p^{s-1} \, \alpha_s \, \sigma^s( b_{\,p^{n-s}-1})\,.
\]
Let's consider the last congruence for $n=kh+i$ where $0 \le i < h$ and divide it by $p^{k(h-1)+i}$: 
\[
p^k \quad | \quad \frac{b_{p^{kh+i}-1}}{p^{k(h-1)+i}} \- \sum_{s=1}^h \frac{\alpha_s}{p} \, \, \frac{\sigma^s(b_{\,p^{kh+i-s}-1})}{p^{k(h-1)+i-s}}\,.
\]
With the notation $\beta_n=b_{p^n-1}/p^{n-\lfloor\frac nh \rfloor}$ we rewrite the last congruence as
\[
p^k \quad | \quad \beta_{kh+i} \- \sum_{s=1}^{h} \frac{\alpha_s}{p} \, p^{\tilde \varepsilon_{is}} \, \sigma^s(\beta_{kh+i-s}) \quad \text{ with }\quad \tilde\varepsilon_{is} \= \begin{cases} 0\,, &s\le i \,,\\ 1\,, & s > i\,.\end{cases} 
\]
Substituting $s=j+1$ and combining all congruences with $i=0,\ldots,h-1$ we obtain
\[
p^k \quad | \quad \begin{pmatrix} \beta_{kh}\quad\;\; \\ \beta_{kh+1}\quad \\ \vdots \\ \beta_{kh+h-2} \\ \beta_{kh+h-1} \end{pmatrix} \- D_k \, \begin{pmatrix} \alpha_1/p \\ \alpha_2/p \\ \vdots \\ \alpha_{h-1}/p \\ \alpha_h \end{pmatrix}\,,
\]
where $D_k$ is the matrix defined in the statement of part~(iii). By~(i) all $\beta_n \in R$, hence all entries of $D_k$ belong to $R$. Therefore it only remains to check that $p \not | \det(D_k)$.
The entries of $D_k$ are 0 modulo $p$ when $\varepsilon_{ij}=1$, that is when $i \le j < h-1$. It means that modulo $p$ the determinant is congruent (up to sign) to the product of the entries under the main diagonal, which are all equal to $\sigma^i(\beta_{kh})$ for $i=1,\ldots,h-1$, times the upper-right entry, which is equal to $\sigma^{h}(\beta_{(k-1)h})$. Using~(ii) we obtain
\[\bal
\det D_k & \equiv (-1)^{h-1} \, \prod_{i=1}^{h-1} \sigma^i(\beta_{kh}) \; \cdot \sigma^h( \beta_{(k-1)h}) \\
& \equiv (-1)^{h-1} \, \prod_{i=1}^{h-1}\prod_{j=0}^{k-1} \sigma^{i+jh}(\beta_{h}) \; \cdot \prod_{j=0}^{k-2}\sigma^{h(1+j)}( \beta_{h}) \= (-1)^{h-1} \,\prod_{s=1}^{kh-1}\sigma^s(\beta_h) \mod p\,. 
\eal\]
\end{proof}

\section{Applications and examples}\label{sec:applications}

We shall now demonstrate how Theorems~\ref{cong_fg_thm} and~\ref{P_formulas_thm} work in three situations. The result of \S~\ref{sec:Lfun} below (Proposition~3) is essentially contained in~\cite[\S 6]{Ho69}. We included this section to demonstrate that our criterion of integrality works particularly well for formal group laws attached to L-functions. In \S~\ref{sec:AM} we discuss Artin--Mazur formal group laws of hypersurfaces. This is the context in which $p$-sequences were invented. In \S~\ref{sec:HG} we extend the results of~\cite{Ho72_HG} to a broad class of hypergeometric formal group laws. We follow the ideas of~\cite{Ho72_HG}, but application of our main theorems allows to make the proofs shorter and hopefully more transparent.

\subsection{Formal group laws attached to L-functions}\label{sec:Lfun}
 
For the purposes of this paper, an \emph{L-function} is a Dirichlet series  
\[
L(s) \= \sum_{n=1}^{\infty}\frac{a_n}{n^s}
\]
with integral coefficients $a_n \in \Z$, which has an Euler product    
\[
L(s) \= \prod_{p \text{ prime }} \mathcal{P}_p(p^{-s})^{-1} \,,
\]
where $\mathcal{P}_p(T)\in 1+T \, \Z[T]$ is a polynomial for every prime number $p$. The Euler product is understood formally, that is we don't care about convergence and the sequence $\{ a_n ; n \ge 1\}$ is determined by the rules:
\be{Lfun_coeffs}\bal
&a_{mn}=a_m a_n \quad\text{ when }\quad(m,n)=1\,,\\
&a_{p^s} \+ \gamma_1(p) \, a_{p^{s-1}} \+ \ldots \+ \gamma_{d(p)}(p) \, a_{p^{s-d}}=0 \quad\text{ where }\quad \mathcal{P}_p(T)\=1+\sum_{i=1}^{d(p)} \gamma_i(p) T^i\,.
\eal\ee
(By convention, we assume $a_k=0$ when $k \not\in \Z_{\ge 1}$.)

 To an L-function one associates an one-dimensional formal group law over $\Q$ by the formula
\be{Lfun_fg_def}
F_L(x,y)=f^{-1}(f(x)+f(y)) \,,  \qquad  f(x)=\sum_{n=1}^{\infty} \frac{a_n}n \, x^n\,.
\ee  
We shall now prove

\begin{proposition}\label{Lfun_thm} For a fixed prime $p$, we have $F \in \Z_{(p)}\lb x,y \rb$ if and only if 
\be{Lfun_fg_cong}
Q(T) := p \, \mathcal{P}_p\Bigl(\frac {T}p \Bigr) \;\in \; p + T \, \Z[T] \,.
\ee

Assume that~\eqref{Lfun_fg_cong} holds. If $Q \= 0 \mod p\;$ then $F$ is strictly isomorphic to $\mathbb{G}_a$ over $\Z_{(p)}$, and otherwise the local characteristic polynomial $\Psi_F(T)$ at $p$ is equal to the unique monic Eisenstein factor of $Q(T)$ in the ring $\Z_p[T]$ of polynomials with $p$-adic integral coefficients and the height $h = \deg \Psi_F$ is equal to the highest power of $T$ that divides $\overline{Q}(T) = Q(T) \mod p$.  
\end{proposition}

The following lemma enables us to describe the $p$-sequence associated to the (shifted) sequence of coefficients of an L-function.

\begin{lemma*}
\label{pseq} Let $\{ a_n ; n \ge 1\}$ be a sequence of integers with $a_1=1$. Fix a prime number $p$. Let $\{ c_n ; n \ge 0\}$ be the $p$-sequence associated to the shifted sequence $\{ b_n = a_{n+1} ; n \ge 0 \}$. 
\begin{itemize}
\item[(i)] We have
\be{Lfun_mult_prop}
a_{m p^k} = a_m a_{p^k} \quad \text{ for all  } m \text{ not divisible by } p  \text{ and }
k \ge 0 \ee
if and only if 
\be{Lfun_p_seq}
c_{m p^k -1} = 0  \qquad \text{ for all  } m>1 \text{ not divisible by } p  \text{ and }
k > 0\,.
\ee

\item[(ii)] There is a recurrence relation of the form
\be{rec_rel}
a_{p^{k}} \+ \gamma_1 \, a_{p^{k-1}} \+ \ldots \+ \gamma_d \, a_{p^{k-d}} \= 0 \quad \text{ for all } k \ge 0 
\ee
with some $d \ge 1$ and integers $\gamma_1, \ldots, \gamma_d$ (where, by convention,  $a_n=0$ when $n \not \in \Z$) if and only if
\[
c_{p^i-1} \= \begin{cases} -\gamma_{i} \,, &1 \le i \le d \,, \\ \quad 0\,, & i > d\,. \end{cases}
\]
\end{itemize}
\end{lemma*}
\begin{proof} (i) $\Rightarrow$ First we prove that~\eqref{Lfun_p_seq} follows from~\eqref{Lfun_mult_prop}. Take any $m>1$ not divisible by $p$ and $k > 0$. We have
\be{rl1}
b_{mp^k-1} \= c_{p-1}\, b_{mp^{k-1}-1} \+ \ldots \+ c_{p^k-1} \, b_{m-1} \+ \sum_{\begin{tiny}\bal & m = m' * m'' \\ &m'>1 , m'' \ge 0\eal\end{tiny}} c_{m' p^k -1} \, b_{m''}\,,
\ee
where the sum runs over all decompositions $m=m'\+p^{\ell(m')}m''$ of the $p$-adic expansion of $m$ into a concatenation of $p$-adic expansions of two non-negative integers, of which $m'$ is actually positive because $k>0$. The condition $m'>1$ is then satisfied automatically due to the fact that $m>1$ is not divisible by $p$. Since $b_{mp^i-1}=b_{p^i-1}\,b_{m-1}$ for $i \ge 0$ due to~\eqref{Lfun_mult_prop}, subtracting from~\eqref{rl1} the equality $b_{p^k-1}=c_{p-1}\,b_{p^{k-1}-1}+\ldots+c_{p^k-1}$ multiplied by $b_{m-1}$ yields
\[
\sum_{\begin{tiny}\bal m = m' * m'' \\ m'>1, m'' \ge 0 \eal\end{tiny}} c_{m' p^k -1} \, b_{m''} \= 0\,.
\]
If $m$ had one $p$-adic digit the sum here would have just one term $c_{mp^k-1} b_0 = c_{mp^k-1}$, so we get $c_{mp^k-1}=0$. Now \eqref{Lfun_p_seq} follows by induction on the length $\ell(m)$ of the $p$-adic expansion of $m$.

$\Leftarrow$ Suppose now that~\eqref{Lfun_p_seq} holds. When $m=1$ or $k=0$ ~\eqref{Lfun_mult_prop} holds automatically. Take any $m>1$ not divisible by $p$ and $k>0$ and consider~\eqref{rl1} once again. The rightmost sum then vanishes due to~\eqref{Lfun_p_seq} since $m' \equiv m \ne 0 \mod p$ for each term. Therefore we get
\[
b_{mp^k-1} \= c_{p-1}\, b_{mp^{k-1}-1} \+ \ldots \+ c_{p^k-1} \, b_{m-1}\,.
\] 
If $k=1$ we have $b_{mp-1}=c_{p-1}b_{m-1}=b_{p-1}b_{m-1}$. For $k>1$ we proceed by induction:
\[
b_{mp^k-1} \= ( c_{p-1}\, b_{p^{k-1}-1} \+ \ldots \+ c_{p^k-1} )\, b_{m-1} \= b_{p^k-1} \, b_{m-1}\,,
\] 
which proves~\eqref{Lfun_mult_prop}. 

(ii) follows immediately from the equalities $a_{p^k} \= c_{p-1} a_{p^{k-1}} \+ c_{p^k-1} a_{p^{k-2}} \+ \ldots \+ c_{p^k-1}$ for every $k \ge 1$ and $a_1=1$.
\end{proof}

\begin{proof}[Proof of Proposition~\ref{Lfun_thm}.]  Let $\{ c_n; n \ge 0\}$ be the $p$-sequence associated with $\{ b_n=a_{n+1}; n \ge 0\}$. By~\eqref{Lfun_coeffs} and part~(i) in the preceding lemma we have $c_{mp^k-1}=0$ for every $m>1$ not divisible by $p$ and any $k>0$. By~\eqref{Lfun_coeffs} and part~(ii) in the preceding lemma we have $c_{p^k-1}=-\gamma_k(p)$ for $k \ge 1$. Hence condition~\eqref{c_cong_0} in Theorem~\ref{cong_fg_thm} is equivalent to $p^{k-1}|\gamma_k(p)$ for $k \ge 1$, which is in turn equivalent to~\eqref{Lfun_fg_cong} because $Q(T)=p \+ \sum_{i=1}^{d(p)}\frac{\gamma_i(p)}{p^{i-1}}T^i$. The characteristic polynomial is the unique Eisenstein factor of $Q(T)$ by the construction reminded at the beginning of Section~\ref{sec:local_invariants}: we have $Q=\eta_{v}$ for the sequence $v=\{v_i = \frac{c_{p^i-1}}{p^{i-1}}; i \ge 1\}$.

\end{proof}

\subsection{Artin--Mazur formal groups}\label{sec:AM}

In~\cite{AM77} Artin and Mazur associated formal groups to cohomology groups of algebraic schemes. In~\cite{St87} Stienstra computed explicit coordinalizations of Artin--Mazur functors related to the middle cohomology of complete intersections and double coverings of projective spaces. Stienstra's formal group laws are integral over the base ring and coefficients of their logarithms are always given as coefficients of powers of a  polynomial (see Theorems~1 and~2 in \emph{loc. cit.}). 

\begin{proposition}\label{Dwork_cong_prop}
Let $R$ be a characteristic~0 ring and $V(x) \in R[x_1^{\pm},\ldots,x_m^{\pm}]$ be a Laurent polynomial. Assume that Newton polytope $\Delta(V) \subset \R^m$ contains a unique internal integral point $\{ w \} = \Delta(V)^{\circ} \cap \Z^m$. Consider the $R$-valued sequence
\[
b_n \= \text{ the coefficient of } x^{n w} \text{ in } V(x)^n\,. 
\]
Assume that $R$ equipped with a $p$th power Frobenius endomorphism and let $\{ c_n ; n \ge 0\}$ be the $p$-sequence attached to $\{ b_n ; n \ge 0\}$. Then one has 
\be{Laurent_poly_cong}
c_n \in p^{\ell(n)-1}R
\ee
for every $n \ge 0$.
\end{proposition}

The proof was essentially given in~\cite[Lemma~1]{MV} and the generalization to rings with Frobenius endomorphism stated here is straightforward. Notice that congruences~\eqref{Laurent_poly_cong} are stronger than~\eqref{c_cong_0}. Therefore it follows from Theorem~\ref{cong_fg_thm} that 

{\bf Corollary.} \emph{The formal group law
\[
F_V(x,y) \= f^{-1}(f(x)+f(y))\,,\qquad f(x)\=\sum_{n=1}^{\infty} \frac{b_{n-1}}{n} x^n
\] 
has coefficients in $R\otimes \Z_{(p)}$ whenever we can equip $R$ with a $p$th power Frobenius endomorphism.}

According to Stienstra, in case $V(x)$ is a homogeneous polynomial this formal group law is a coordinalization of the Artin--Mazur formal group attached to the middle cohomology of the hypersurface of zeroes of $V(x)$. Integrality of such formal group laws was established in~\cite{St87}, but our case is more general and methods are elementary. 

In~\cite{St87L} Stienstra proved a generalization of the Atkin and Swinnerton-Dyer congruences when $R$ is of finite type over $\Z$. Such congruences allow one to establish isomorphism of Artin--Mazur formal groups and formal groups attached to respective L-functions. Let us demonstrate how Theorem~\ref{P_formulas_thm} works in an example with a double covering. Consider a K3 surface $\mathfrak X$ (called $\mathcal A'$ in~\cite{BS85}) given by  
\[
Y^2 \= T_0 \, T_1 \, T_2 \, (T_1 \- T_2 )(T_1 T_2 \- T_0^2)\,.
\]
According to a computation in \emph{loc. cit.}, for every $p \ne 2$ the respective Euler factor of the L-function associated to $H^2(\mathfrak X)$ has the shape
\[
\mathcal{P}_p(T) \= (1-p \,T)^{20} (1 \+ A_p \, T \+ (-1)^{\frac{p-1}2} p^2 \, T^2)
\]  
where 
\[
A_p \= \begin{cases} \quad 0 \,, &\text{ if } p \equiv 3 \mod 4 \,,\\
2p - 4 a^2 \,, &\text{ if } p \equiv 1 \mod 4 \;\; (p = a^2 + 4 b^2\,, a,b \in \Z) \,.\\
\end{cases}
\]
The conditions of Proposition~\ref{Lfun_thm} are satisfied since
\[
Q(T) \= p \, \mathcal{P}_p(T/p) \= (1-T)^{20} (p \+ A_p \, T \+ (-1)^{\frac{p-1}2} p \, T^2) \;\in\; \Z[T]\,.
\]
By Proposition~\ref{Lfun_thm} the local characteristic polynomial of this formal group law at $p$ is given by
\be{K3_char_p}
\Psi(T) \= \begin{cases} p \,, &\text{ if } p \equiv 3 \mod 4 \,,\\
p \- \alpha \, T \,, &\text{ if } p \equiv 1 \mod 4 \;\; (\alpha^2 + A_p \, \alpha + p^2 \= 0, \alpha \, \in \Z_p^{\times})\,.\\ \end{cases}
\ee
On the other hand, by~\cite[Theorem~2]{St87} the coefficients of the logarithm of a coordinalization of the Artin--Mazur formal group associated to $H^2(\mathfrak X)$ are given by
\[\bal
b_n & \= \begin{cases}
0\,, & n \text{ odd,}\\
\text{ the coefficients of } T_0^nT_1^nT_2^n \text{ in } (T_0T_1T_2)^{n/2}(T_1 \- T_2 )^{n/2}(T_1 T_2 \- T_0^2)^{n/2} \,, & n \text{ even,}\\
\end{cases} \\
& \= \begin{cases}
0\,, & 4 \not | n \,,\\
\binom{n/2}{n/4}^2 \,, & 4 \;\; | n \,.\\
\end{cases}
\eal\]
We apply formula~\eqref{unit_root} to obtain a $p$-adic analytic formula for the coefficient $\alpha \in \Z_p^{\times}$ in~\eqref{K3_char_p} when $p \equiv 1 \mod 4$:
\[
\alpha \= \lim_{k \to \infty} \frac{b_{p^{k}-1}}{b_{p^{k-1}-1}} \= \lim_{k \to \infty}  \frac{\Gamma_p\bigl(\tfrac{p^{k}-1}2+1\bigr)^2}{\Gamma_p\bigl(\tfrac{p^k-1}4+1\bigr)^4} \= \frac{\Gamma_p(\tfrac12)^2}{\Gamma_p(\tfrac34)^4} \= \frac{(-1)^{\tfrac{p+1}2}}{\Gamma_p(\tfrac34)^4} \= - \Gamma_p(\tfrac14)^4\,,
\]
where $\Gamma_p(.)$ is the $p$-adic gamma function (see \cite[\S IV.2]{Kob}). It follows that
\[
A_p \= \- \alpha \- p \, \alpha^{-1} \= \Gamma_p(\tfrac14)^4 \+ p \, \Gamma_p(\tfrac34)^4\,.
\]

\subsection{Hypergeometric formal group laws}\label{sec:HG}

Consider a finite system of integral weights on positive integers $\gamma = \{ \gamma_\nu ; \nu \ge 1\}$. By this we mean that each $\gamma_\nu \in \Z$ and $\gamma_\nu=0$ for all but finitely many $\nu$. We assume that $\gamma$ satisfies the condition 
\be{HG_cond}
\sum_{\nu \ge 1} \nu \, \gamma_\nu \= 0\,.
\ee
Following~\cite{RV03}, we associate to $\gamma$ the sequence of rational numbers $u_n \= \prod_{\nu \ge 1} (\nu \, n)!^{\gamma_\nu}$, $n \ge 0$. One has $\ord_p(u_n) = \sum_{i=1}^{\infty} \L(\frac{n}{p^i})$ where 
\[
\L(t) \= \sum_{\nu} \gamma_\nu \lfloor \nu \, t \rfloor 
\] 
is the associated~\emph{Landau function}. This function is right-continuous and locally constant. Condition~\eqref{HG_cond} implies that $\L(t)$ is periodic with period~1. It is named after Emil Landau who observed that the sequence $\{u_n ; n \ge 0\}$ is integral if and only if 
\be{int_Landau}
\L(t) \ge 0 \quad\text{ for all }\quad t\,.
\ee
Let $N \= l.c.m. \{ \nu \,|\, \gamma_\nu \ne 0\}$. Assuming the integrality of $\{ u_n; n \ge 0\}$, we will give a criterion of integrality of the rational formal group law given by
\[
F_{\gamma}(x,y) \= f^{-1}(f(x)+f(y))\,,\qquad  f(x) \= \sum_{n=0}^{\infty} u_{n} \frac{x^{Nn+1}}{Nn+1} \= \int \sum u_n x^{Nn} dx \,.\]

\begin{proposition}\label{hg_prop} Suppose the system of weights $\gamma = \{ \gamma_\nu ; \nu \ge 1\}$ satisfies conditions~\eqref{HG_cond} and~\eqref{int_Landau}. 

Let $p > N$ be a prime. Let $d = \min\{ m > 0 \;|\; p^m \equiv 1 \mod N \}$ be the order of $p$ in the multiplicative group of invertible residues modulo $N$. For each $0 \le a < d$ let $1 \le m_a < N$ be the unique number such that $m_a p^a \equiv 1 \mod N$. Define $\lambda_a := \ord_p \bigl(u_{\frac{m_ap^a-1}{N}} \bigr)$.

\begin{itemize}
\item[(i)]  We have $\lambda_a \=  \sum_{i=1}^{a} \L\bigl( \frac{p^{d-i}-1}{N}\bigr)$ for $0 \le a \le d-1$. 

\item[(ii)] The coefficients of $F_{\gamma}$ are $p$-integral if and only if $\lambda_a \ge a$ for each $1 \le a \le d-1$.

\item[(iii)] Assume~(ii) is satisfied. If $\lambda_{d-1} = d-1$, the height of $F_{\gamma}$ at $p$ equals $d$ and the local characteristic polynomial is given by
\be{xi_expr}
\Psi_{F_{\gamma}}(T) \= p \- \xi \, T^d \quad \text{ with } \quad \xi  \= (-1)^{(\gamma_2+1)d-1} \,\prod_{\nu} \prod_{a=0}^{d-1} \Gamma_p\Bigl( \bigl\{1 \- \frac{ \nu m_a}{N}\bigr\}\Bigr)^{\gamma_\nu} \,,
\ee
where $\{ x\} = x \- \lfloor x \rfloor$ is the fractional part of $x$.
If $\lambda_{d-1} \ge d$ then the height at $p$ is infinite. 
\end{itemize}
\end{proposition}

Notice that due to the periodicity of the Landau function and formula~(i) the numbers $\lambda_a$ depend only on the residue class $\; p \mod N$. Hence the property of $p$-integrality~(ii) and the height depend only on $\; p \mod N$.

\begin{proof}[Proof of Proposition~5.] We write $f(x) \= \sum_{n=1}^{\infty} b_{n-1} \frac{x^n}{n}$ where $b_n \= u_{\frac{n}{N}}$ if $N \;|\, n$ and $b_n=0$ otherwise. Note that jumps of $\L(t)$ happen at rational points whose denominators divide $N$. Therefore $\L(t)=0$ for $0 \le t < \frac1N$. Using this fact and periodicity of $\L(t)$ we see that
\[\bal
\ord_p \bigl(b_{m_ap^a-1} \bigr) & \= \sum_{i=1}^{\infty} \L\bigl( \frac{m_ap^a-1}{Np^i}\bigr) \= \sum_{i=1}^{a} \L\bigl( \frac{m_ap^a-1}{Np^k}\bigr) \\
& ( p^i < m_ap^a \Rightarrow i \le a \;\text{ since }\; m_a < N < p )\\
&\= \sum_{i=1}^{a} \L\bigl( \frac{m_ap^{a-i}-1}{N}\bigr) \qquad (\text{because }\; \bigl\lfloor \frac{m_ap^a-1}{p^i}\bigr\rfloor \= m_ap^{a-i}-1 )\\
&\= \sum_{i=1}^{a} \L\bigl( \frac{p^{d-i}-1}{N}\bigr) \qquad (\text{because }\;m_a \equiv p^{d-a} \mod N )\,,\\
\eal\] 
which proves~(i).

Let $\{ c_n\}$ be the $p$-sequence associated to $\{b_n\}$. It is easy to see that $c_n = 0$ when $N \not| n$. We also observe that $c_{m_ap^a-1}=b_{m_ap^a-1}$ for each $1 \le a < d$. Indeed, since $m_a$ is the smallest positive integer with the property $m_ap^a \equiv 1 \mod N$ it is impossible to divide the $p$-adic expansion of $m_ap^a-1$ into two parts so that each part represents a number divisible by $N$. It now follows from Theorem~1 that the condition in~(ii) is necessary for $p$-integrality of $F_\gamma$.

Next, assume that $\lambda_a \ge a$ for $1 \le a < d$. We shall show that $F_\gamma$ is $p$-integral and simultaneously prove~(iii). Same arguments as in the above proof of~(i) yield $\ord_p \bigl(b_{p^d-1} \bigr) \= \sum_{i=1}^{d-1} \L\bigl( \frac{p^d-1}{Np^i}\bigr) \= \sum_{i=1}^{d-1} \L\bigl( \frac{p^{d-i}-1}{N}\bigr) \= \ord_p \bigl(b_{m_{d-1}p^{d-1}-1} \bigr) = \lambda_{d-1}$, which is $\ge d-1$ by our assumption. 

Let $k=a+sd$ with $0 \le a < d$, $s \ge 0$ and assume that $m \equiv m_a \mod N$. For the rest of the proof we will need the following two observations:
\be{ordpb_ineq}
\ord_p \bigl(b_{mp^k-1} \bigr) \ge \ord_p \bigl(b_{m_a p^k-1} \bigr)
\ee
and
\be{ordpb_shift}
\ord_p \bigl(b_{mp^k-1} \bigr) \= s \, \ord_p \bigl(b_{p^d-1} \bigr) \+ \ord_p \bigl(b_{mp^a-1} \bigr)\,. 
\ee
We prove~\eqref{ordpb_ineq} as follows
\[\bal
\ord_p \bigl(b_{mp^k-1} \bigr) & \= \sum_{i=1}^{k+\ell(m)-1} \L\bigl( \frac{mp^{k}-1}{Np^i}\bigr) \;\overset{ \text{by \eqref{int_Landau}}}\ge\; \sum_{i=1}^{k} \L\bigl( \frac{mp^{k}-1}{Np^i}\bigr) \= \sum_{i=1}^{k} \L\bigl( \frac{mp^{k-i}-1}{N}\bigr)\\
&   \overset{\begin{small}m \equiv m_a \,{\rm mod}\, N\end{small}} \= \sum_{i=1}^{k} \L\bigl( \frac{m_ap^{k-i}-1}{N}\bigr) \= \sum_{i=1}^{k} \L\bigl( \frac{m_ap^{k}-1}{Np^i}\bigr) \= \ord_p(b_{m_a p^{k}-1})\,.
\eal\]
For~\eqref{ordpb_shift} we observe that
\[\bal
\ord_p \bigl(b_{mp^{k}-1} \bigr) \- \ord_p \bigl(b_{mp^{k-d}-1} \bigr) &\= \sum_{i=1}^{k+\ell(m)-1} \L\bigl( \frac{m p^{k-i}-1}{N}\bigr) \-  \sum_{j=1}^{k-d+\ell(m)-1} \L\bigl( \frac{m p^{k-d-j}-1}{N}\bigr) \\
&\overset{i=j+d}\= \quad \sum_{i=1}^{d} \L\bigl( \frac{m p^{k-i}-1}{N}\bigr) \= \sum_{i=1}^{d} \L\bigl( \frac{ p^{d-i}-1}{N}\bigr) \= \ord_p(b_{p^d-1})\,.
\eal\] 

When $\ord_p(b_{p^d-1}) \ge d$ then~\eqref{ordpb_ineq}, \eqref{ordpb_shift} and the assumption $\lambda_a \ge a$ yield 
\[\ord_p(b_{m p^k-1}) \ge \ord_p(b_{m_a p^k-1}) = s d + \ord_p(b_{m_a p^a-1}) \ge s d + a = k\,,\]
hence $f(x) \in \Z_{(p)}\lb x\rb$ and $F_\gamma$  is strictly isomorphic to $x+y$ over $\Z_{(p)}$. 

It remains to consider the case $\ord_p(b_{p^d-1})=d-1$. Then by~\eqref{ordpb_shift} we have $\ord_p(b_{p^{sd}-1}) = s \, \ord_p(b_{p^d-1}) = s(d-1)$. If $F_{\gamma}$ were $p$-integral then by (i) in Theorem~2 the height would equal $h=d$ and by~(iv) in Theorem~2 the characteristic polynomial would equal $\Psi_p(T)=p - \xi \, T^d$ with $\xi \in \Z_p^{\times}$ is such that $\xi \equiv \frac1{p^{d-1}} \, \frac{b_{p^{sd}-1}}{b_{p^{(s-1)d}-1}} \mod p^s$ for each $s \ge 1$. We shall show that, more generally, for $k \ge d$ and $m$ such that $mp^k \equiv 1 \mod N$ one has 
\be{padic_ratio_formula}
\frac{b_{m p^k-1}}{p^{d-1} b_{m p^{k-d}-1} } \equiv \xi \mod p^{k+1-d}
\ee
with the $p$-adic unit $\xi$ given in~\eqref{xi_expr}. Notice that~\eqref{padic_ratio_formula} implies that $\ord_p  \bigl(b_{m p^{k}-1} - \xi p^{d-1} b_{m p^{k-d}-1}\bigr) \ge (k+1-d) + (d-1) \+ \ord_p(b_{m p^{k-d}-1}) \ge k$ for all $k \ge d$. Since for $0 \le k < d$ we also have $\ord_p(b_{mp^k-1}) \ge \ord_p(b_{m_k p^k-1}) \ge k$ by~\eqref{ordpb_ineq} and the assumption $\lambda_a \ge a$, we obtain that   
\[
f(x) \- \frac1p \, \xi \, f(x^{p^d}) \in \Z_p\lb x \rb\,,
\]
and $p$-integrality of $F_{\gamma}$ then follows from Hazewinkel's functional equation lemma.

The proof of~\eqref{padic_ratio_formula}  is a routine calculation in $p$-adic analysis. The $p$-adic gamma function is defined so that $\frac{n!}{p^{\lfloor n/p \rfloor} \lfloor n/p \rfloor!} = (-1)^{n+1}\Gamma_p(n+1)$ for any $n \ge 1$. Observe that one has $\lfloor \frac1{p^i}\lfloor \frac{n}{p}\rfloor\rfloor = \lfloor \frac{n}{p^{i+1}}\rfloor$ for each $i \ge 1$. With $n=\frac{mp^k-1}{N}$ one has $\lfloor \frac{\nu n}{p^i}\rfloor = \frac{mp^{k-i}}{N/\nu} - \langle\frac{\nu m_i}{N}\rangle$ whenever $i<d$, where $\langle x \rangle = \{x\} = x - \lfloor x \rfloor$ when $x \not\in \Z$ and $\langle x \rangle = 1$ when $x \in \Z$. We apply the above formula for the ratio of factorials $d$ times with $0 \le i < d$ and get 
\[
\frac{(\nu n)!}{p^* \lfloor \nu n/p^d \rfloor!} \= (-1)^{\varepsilon_\nu} \prod_{i=0}^{d-1} \Gamma_p\bigl(1+ \frac{mp^{k-i}}{N/\nu}-\langle  \frac{\nu m_i}{N} \rangle\bigr) \equiv (-1)^{\varepsilon_\nu} \prod_{i=0}^{d-1} \Gamma_p\bigl(1 - \langle  \frac{\nu m_i}{N} \rangle\bigr) \mod p^{k+1-d}\,,
\]
where
\[
\varepsilon_\nu \= \sum_{i=0}^{d-1} \bigl(1+ \frac{mp^{k-i}}{N/\nu}-\langle  \frac{\nu m_i}{N} \rangle\bigr)
\= d \+ \frac{\nu \, m p^k}{N}  \sum_{i=0}^{d-1}\frac1{p^i} - \sum_{i=0}^{d-1}\langle  \frac{\nu m_i}{N} \rangle
\]
and we don't need to care about the power of $p$ on the left in this formula since we already know they will sum up to $d-1$ in~\eqref{padic_ratio_formula}. Note also that $1-\langle x \rangle = \{1-x\}$. Condition~\eqref{HG_cond} now implies that $\sum_{\nu} \gamma_{\nu}\varepsilon_\nu$ is independent of the pair $(m,k)$, so we can take it to be $(1,d)$ and get
\[\bal
\sum_{\nu} \gamma_{\nu}\varepsilon_\nu \= \bigl(\sum_{\nu} \gamma_\nu\bigr)d \+ \sum_{\nu} \gamma_{\nu} \sum_{i=0}^{d-1} \Bigl\lfloor\frac{\nu(p^d-1)}{N \, p^i} \Bigr\rfloor \= \bigl(\sum_{\nu} \gamma_\nu\bigr)d \+\sum_{i=0}^{d-1} \L\bigl(\frac{p^d-1}N\bigr) \\
\= (\sum_{\nu} \gamma_\nu\bigr)d \+ \ord_p(b_{p^d-1}) \= (\sum_{\nu} \gamma_\nu\bigr)d \+ d-1 \;\equiv\; (\gamma_2 \+ 1\bigr)d \-1 \mod 2\,,
\eal\]
which proves our claim.  
\end{proof}

Next we shall establish a criterion of $p$-integrality of $F_{\gamma}$ for all residue classes $p \mod N$. 

\begin{proposition}\label{int_all_p_prop} The coefficients of the formal group law $F_\gamma$ are $p$-integral for all but finitely many primes $p$ if and only if the Landau function satisfies
\be{int_all_p_1}
\L\bigl( \frac {i-1} N \bigr) \ge 1 \quad\text{ for every }\quad 2 \le i \le N-1 \quad \text{ with }\quad (i,N)=1
\ee
If~\eqref{int_all_p_1} is satisfied then $F_\gamma$ is $p$-integral for all $p>N$. 
\end{proposition}
\begin{proof} Recall that for a prime $p > N$ the property of $F_\gamma$ to be $p$-integral depends only on $p \mod N$. Assume that $F_\gamma$ is $p$-integral for all but finitely many primes. For every $i$ as in~\eqref{int_all_p} pick $p>N$ such that $p^{-1} \equiv i \mod N$. Then by~(ii) and~(i) with $a=1$ we have $1 \ge \lambda_1 = \L((p^{d-1}-1)/N) = L((i-1)/N)$. Conversely, assuming that~\eqref{int_all_p_1} is satisfied we have $\lambda_a \ge a$ for every $p>N$ and every $1 \le a < d(p)$ because every term in the sum~(i) for $\lambda_a$ is $\ge 1$. 
\end{proof}

Notice that~\eqref{int_all_p_1} is equivalent to
\be{int_all_p}
\L\bigl( \frac i N \bigr) \ge \gamma_N + 1 \quad\text{ for every }\quad 2 \le i \le N-1 \quad \text{ with }\quad (i,N)=1
\ee
because of the identity $\L(\frac{i}{N}) = \L(\frac{i-1}{N}) + \gamma_N$ for $i$ coprime to $N$. Indeed, for every $\nu | N$, $\nu \ne N$ we have $\lfloor \frac{\nu i}{N} \rfloor = \lfloor \frac{\nu(i-1)}{N} \rfloor$ whenever $(i,N)=1$.

Let us consider a couple of examples with $N=15$. Here is a table of heights computed using Proposition~\ref{hg_prop}. We write '--' when there is no integrality, and '*' means that $\L(\frac{i-1}{N})=0$, that is condition~\eqref{int_all_p_1} breaks for this $i$.

\begin{center}
\begin{tabular}{c|cccccccc}
$i$ &1&2&4&7&8&11&13&14\\
\hline
&&&&&&&&\\
$u_n = \frac{(15n)!n!}{(3n)!^2(5n)!^2}$ &1&4&2&--*&4&--*&--*&2\\
&&&&&&&&\\
$u_n = \frac{(5n)!}{n!^2(3n)!}$ &1&--$*$&2&$\infty$&--&2&$\infty$&$\infty$\\
\end{tabular}
\end{center}
Since Landau functions add when one adds two systems of weights, we see that the product of two factorial ratios in the table $u_n=\frac{(15n)!}{n!(3n)!^3(5n)!}$ will give a formal group integral at all $p>15$. This is because there are no two $*$'s at the same column.

\bigskip

In this section we extended results, which Taira Honda obtained for sequences
\be{un_Honda}
u_n \= \prod_{k=1}^r\binom{\alpha_k+n}{n} \= \frac{(\alpha_1)_n \ldots (\alpha_r)_n}{n!^r} \,, \qquad \{\alpha_1,\ldots, \alpha_r\} \subseteq \{\frac1N,\ldots,\frac{N-1}N\}\,,
\ee
where $(x)_n=\prod_{i=0}^{n-1}(x+i)$ is the Pochhammer symbol, to integral ratios of factorials $u_n = \prod (\nu\, n)!^{\gamma_\nu}$. Namely, Theorem~1 of~\cite{Ho72_HG} is stated for our sequences $\{u_n\}$ in Proposition 5 (ii)-(iii), Theorems~2 and~3 of~\cite{Ho72_HG} in Proposition~5 (i) and Proposition~6 respectively. Honda's sequence~\eqref{un_Honda} can be written as an integral ratio of factorials (times $C^n$ for a constant $C$ not divisible by any prime $p>N$) whenever the set $\{ e^{2\pi i \alpha_1},\ldots, e^{2\pi i \alpha_r}\}$ is closed under Galois conjugation. Conversely, an integral ratio of factorials can be written via Pochhammer symbols (see~\cite{RV03}). Namely, one has
\be{u_via_Pochhammer}
 \prod_\nu (\nu\, n)!^{\gamma_\nu} \= C^n \frac{(\alpha_1)_n \ldots (\alpha_r)_n}{(\beta_1)_n \ldots (\beta_r)_n} \,,\quad 
\ee
where $C=\prod_{\nu} \nu^{\nu \, \gamma_\nu}$, and $\{\alpha_i\}$, $\{\beta_j\}$ are two distinct sets of rational numbers related to the system of weights $\gamma$ via the formula
\[
\prod_\nu (t^\nu -1)^{\gamma_\nu} \= \prod_{j=1}^r (t - e^{2 \pi i \alpha_j}) / \prod_{k=1}^r (t - e^{2 \pi i \beta_k})\,.
\]
Hence our sequences are of type~\eqref{un_Honda} whenever $\beta_j = 1$ for all $j$. It is very likely that statements similar to Propositions~\ref{hg_prop} and~\ref{int_all_p_prop} also hold for more general ratios of Pochhammer symbols.

\end{document}